\documentclass[a4paper,12pt]{amsart}
\usepackage{verbatim}
\newtheorem{theorem}{Theorem}[section]

\newtheorem{proposition}[theorem]{Proposition}
\newtheorem{corollary}[theorem]{Corollary}

\theoremstyle{definition}
\newtheorem{definition}{Definition}[section]
\newtheorem{example}{Example}[section]

\theoremstyle{remark}
\newtheorem*{remark}{Remark}

\numberwithin{equation}{section}

\def\N{{\mathbb
N}}
\def\R{{\mathbb R}}

\DeclareMathOperator{\spn}{span}
\DeclareMathOperator{\spann}{span}
\DeclareMathOperator{\ran}{ran}

\newcommand{\Es}{E^\ast}

\newcommand{\xast}{x^{\ast}}

\newcommand{\Xast}{X^{\ast}}

\newcommand{\Xastast}{X^{\ast\ast}}

\newcommand{\n}[1]{\left\|#1\right\|}

\newcommand{\us}{u^{\ast}}
\newcommand{\vs}{v^{\ast}}

\newcommand{\Xs}{X^{\ast}}
\newcommand{\xs}{x^{\ast}}
\newcommand{\Xss}{X^{\ast\ast}}
\newcommand{\xss}{x^{\ast\ast}}
\newcommand{\Ys}{Y^{\ast}}
\newcommand{\ys}{y^{\ast}}

\newcommand{\Zs}{Z^{\ast}}
\newcommand{\zs}{z^{\ast}}
\newcommand{\eps}{\varepsilon}

\begin{document}

\title{On duality of diameter 2 properties}
\date{26.04.2013}

\author{Rainis Haller}
\curraddr{}
\thanks{This research was supported by Estonian Targeted Financing Project SF0180039s08.}

\author{Johann Langemets}
\curraddr{}
\thanks{}

\author{M\" art P\~oldvere}
\curraddr{}
\thanks{}



\date{}

\dedicatory{}

\begin{abstract}
It is known that a Banach space has the strong diameter $2$ property
(i.e. every convex combination of slices of the unit ball has diameter $2$)
if and only if the norm on its dual space is octahedral (a notion introduced by Godefroy and Maurey).
We introduce two more versions of octahedrality, which turn out to be dual properties to
the diameter $2$ property and its local version
(i.e., respectively, every relatively weakly open subset and every slice of the unit ball has diameter $2$).
We study stability properties of different types of octahedrality,
which, by duality, provide easier proofs of many known results on diameter $2$ properties.
\end{abstract}

\maketitle
\section{Introduction}\label{sec: 1. Introduction}
All Banach spaces considered in this paper are nontrivial and over the real field. First let us fix some notation. Let $X$ be a Banach space. The closed unit ball of $X$ is denoted by $B_X$ and its unit sphere by $S_X$. The dual space of $X$ is denoted by $X^\ast$. By a \emph{slice} of $B_X$ we mean a set of the form
\[
S(x^*,\alpha)=\{x\in B_X\colon x^*(x)>1-\alpha\},
\]
where $x^*\in S_{X^*}$ and $\alpha>0$.

According to the terminology in \cite{ALN}, a Banach space $X$ has the
\begin{itemize}\label{def: d2p-s}
\item \emph{local diameter $2$ property} if every slice of $B_X$ has diameter $2$;
\item \emph{diameter $2$ property} if every nonempty relatively weakly open subset
of $B_X$ has diameter $2$;
\item \emph{strong diameter $2$ property} if every convex combination of slices of $B_X$ has diameter $2$, i.e.
the diameter of $\sum_{i=1}^n \lambda_i S_i$ is $2$, whenever $n\in\mathbb N$, $\lambda_1,\dotsc,\lambda_n\geq 0$ with $\sum_{i=1}^n\lambda_i=1$, and $S_1,\dotsc,S_n$ are slices of $B_X$.
\end{itemize}

The question whether the three diameter $2$ properties are really different, remained open in \cite{ALN}. However, by now it is known that they are distinguishable. On the one hand, the diameter $2$ property clearly implies the local diameter $2$ property, and the strong diameter $2$ property implies the diameter $2$ property. This follows directly from Bourgain's lemma \cite[Lemma II.1 p.~26]{GGMS}, which asserts that every nonempty relatively weakly open subset of $B_X$ contains some convex combination of slices.

An important consequence of the investigation in \cite{ABL} by Acosta, Becerra~Guerrero and L\'opez~P\'erez is that the strong diameter $2$ property is absent on $p$-sums of Banach spaces for $1<p<\infty$. (The latter result was independently obtained in the Master's Thesis of the second named author, defended at the University of Tartu in June 2012 (see also \cite{HL}).) Since the diameter $2$ property is stable by taking $\ell_p$-sums for all $1\leq p\leq\infty$ \cite{ALN}, this affirms that the strong diameter $2$ property is essentially different from the (local) diameter $2$ property.

On the other side, in a recent preprint \cite{GPZ}, Becerra~Guerrero, L\'opez P\'erez, and Rueda~Zoido constructed a Banach space enjoying the local diameter $2$ property but lacking the $2$ property; moreover, the unit ball of this space contains nonempty relatively weakly open subsets with arbitrarily small diameters.


If $X$ is a dual space, then slices of $B_X$ whose defining functional comes from (the canonical image of) the predual of $X$ are called \emph{weak$^\ast$ slices} of $B_X$. A natural question to ask is whether diameter $2$ properties of a dual space remain the same properties if, instead of all slices or relatively weakly open subsets, one considers only weak$^\ast$ slices or relatively weak$^\ast$ open subsets.

\begin{example}\label{example}
Every convex combination of weak$^\ast$ slices of $B_{C[0,1]^\ast}$ has diameter $2$
(this follows by observing that every weak$^\ast$ slice of $B_{C[0,1]^\ast}$ contains infinitely many different functionals arising via integrating against a measure supported at a singleton);
however, $B_{C[0,1]^\ast}$ has slices with arbitrarily small diameter
(to see this, observe that $C[0,1]^\ast\cong\ell_1([0,1])\oplus_1 C[0,1]^\ast$, and $\ell_1([0,1])$ has the Radon--Nikod\'ym property).
\end{example}
Example~\ref{example} suggests that it makes sense to consider also the weak$^\ast$ versions of the diameter 2 properties.

\begin{definition}\label{def: weak* d2p-s}
Let $X$ be a Banach space. We say that $X^\ast$ has the
\begin{itemize}
\item \emph{weak$^\ast$ local diameter $2$ property} if every weak$^\ast$ slice of $B_{X^\ast}$ has diameter $2$;
\item \emph{weak$^\ast$ diameter $2$ property} if every nonempty relatively weak$^\ast$ open subset
of  $B_{X^\ast}$ has diameter $2$;
\item \emph{weak$^\ast$ strong diameter $2$ property} if every convex combination of weak$^\ast$ slices of $B_{X^\ast}$ has diameter $2$.
\end{itemize}
\end{definition}

The following relationship between the diameter $2$ properties is straightforward to verify.

\begin{proposition}\label{prop: X d2p = X** weak* d2p}
A Banach space $X$ has the local diameter $2$ property (respectively, the diameter $2$ property, the strong diameter $2$ property) if and only if $X^{\ast\ast}$ has the weak$^\ast$ local diameter $2$ property (respectively, weak$^\ast$ diameter $2$ property, weak$^\ast$ strong diameter $2$ property).
\end{proposition}

In the present paper, we study the weak$^\ast$ diameter $2$ properties more deeply. The starting point of our investigations is the following result by Deville (cf. \cite[Proposition~3]{D}).
\begin{proposition}\label{Deville}
If the norm on $X$ is octahedral, then $X^\ast$ has the weak$^\ast$ strong diameter $2$ property.
\end{proposition}

\begin{remark}
Deville's assertion is, in fact, that $\sum_{i=1}^n 1/{n}\, S^\ast_i$ has diameter $2$ whenever $n\in\mathbb N$ and $S^\ast_1,\dotsc,S^\ast_n$ are weak$^\ast$ slices of $B_{X^\ast}$. It is straightforward to verify that the latter is equivalent to the weak$^\ast$ strong diameter $2$ property of $X^\ast$.

Likewise, a Banach space $X$ has the strong diameter $2$ property if (and only if) $\sum_{i=1}^n 1/{n}\, S_i$ has diameter $2$ whenever $n\in\mathbb N$ and $S_1,\dotsc,S_n$ are slices of $B_X$.
\end{remark}

The reverse implication of Proposition~\ref{Deville} stays unproven in \cite{D} (see \cite[Remark (c) after Proposition~3]{D}). However, Godefroy (cf. \cite[p.~12]{G}) marks without an explanation that the norm on a Banach space $X$ is octahedral if and only if $X^\ast$ has the weak$^\ast$ strong diameter $2$ property. In what follows, we present a simple direct proof of this fact (see
Theorem 3.5). An alternative proof can be found in a very recent preprint \cite{BGLPRZoca}.


Let us summarize the results of the paper.

In Section~\ref{sec: 2. Octahedrality}, we introduce  two more octahedrality-type properties of the norm, which correspond to the (weak$^\ast$) local diameter $2$ property and to the (weak$^\ast$) diameter $2$ property, respectively. We also provide equivalent reformulations for different types of octahedrality, which are often more convenient to use.


Relationship between weak$^\ast$ diameter $2$ properties and the corresponding octahedrality properties is established in  Section~\ref{sec: 3. Criteria for weak* d2p-s} (Theorems~\ref{thm: omnibus thm for the weak* ld2P}, \ref{thm: omnibus thm for the weak* d2P}, and \ref{thm: omnibus thm for the weak* Sd2P}). As a consequence of Proposition~\ref{prop: X d2p = X** weak* d2p}, our characterizations of the weak* diameter $2$ properties lead to dual characterizations of the corresponding diameter $2$ properties (Theorems~\ref{thm: omnibus thm for the ld2P}, \ref{thm: omnibus thm for the d2P}, \ref{thm: omnibus thm for the sd2P}). We also show that diameter $2$ properties may be considered as sort of extension properties.

In Section~\ref{sec: 4. Stability of octahedrality}, we study stability properties of different types of octahedrality.
This section is motivated by the idea to provide octahedrality-based approach to known stability results on diameter $2$ properties. We are convinced that in many cases this method is more convenient and preferable.

\section{Octahedrality}\label{sec: 2. Octahedrality}
\begin{definition}[see \cite{G} and \cite{DGZ}, cf. \cite{D}]\label{def: octahedral}
Let $X$ be a Banach space. The norm on $X$ is \emph{octahedral}
if, for every finite-dimensional subspace $E$ of $X$ and every $\eps>\nobreak0$,
there is~a $y\in S_X$ such that
\[
\|x+y\|\geq(1-\eps)\bigl(\|x\|+\|y\|\bigr)\quad\text{for all $x\in E$.}
\]
Whenever it makes no confusion, throughout the paper, spaces whose norm is octahedral,
will also be called octahedral for simplicity.
\end{definition}

Octahedral norms were introduced by Godefroy and Maurey \cite{GM} (see also \cite{G})
in order to characterize Banach spaces containing an isomorphic copy of $\ell_1$.
The connection of octahedral norms to the subject appears probably first in Deville's paper \cite[Proposition~3]{D} (see Proposition~\ref{Deville}). In Theorems~\ref{thm: omnibus thm for the weak* Sd2P} and \ref{thm: omnibus thm for the sd2P} below, we expose the duality between octahedrality and the strong diameter $2$ property.

In order to characterize spaces whose dual has the weak$^\ast$ local diameter $2$ property or the weak$^\ast$ diameter $2$ property,
we introduce two more octahedrality-type properties of the norm.

\begin{definition}\label{def: locally octahedral, weakly octahedral}
Let $X$ be a Banach space. We say that (the norm on) $X$ is
\begin{itemize}
\item
\emph{locally octahedral} if, for every $x\in X$ and every $\eps>0$, there is a $y\in S_X$ such that
\[
\|sx+y\|\geq(1-\eps)\bigl(|s|\|x\|+\|y\|\bigr)\quad\text{for all $s\in\R$;}
\]
\item
\emph{weakly octahedral} if, for every finite-dimensional subspace $E$ of~$X$, every $\xs\in B_{\Xs}$, and every $\eps>\nobreak0$,
there is~a $y\in S_X$ such that
\[
\|x+y\|\geq(1-\eps)\bigl(|\xs(x)|+\|y\|\bigr)\quad\text{for all $x\in E$.}
\]
\end{itemize}
%
\end{definition}

\begin{remark}\label{rem: loc-oct is inf-dimensional}
Clearly, every weakly octahedral Banach space is locally octahedral, and every octahedral Banach space is weakly octahedral.

Note that a locally octahedral Banach space $X$ is infinite-dimensional. Indeed, for a finite-dimensional $X\not=\{0\}$, there exists a weak$^\ast$ slice $S(x,\alpha)$ of $B_{X^\ast}$, whose diameter is less than $\alpha$. If $X$ is locally octahedral, then $\|x\pm y\|\geq 2-\alpha$ for some $y\in S_X$, and therefore $x_1^\ast(x+y)\geq 2-\alpha$ and $x_2^\ast(x-y)\geq 2-\alpha$ for some $x_1^\ast,x_2^\ast\in S_{X^\ast}$. It follows that
\[
x_1^\ast(x),\,x_2^\ast(x),\,x_1^\ast(y),\,-x_2^\ast(y)>1-\alpha,
\]
thus $x_1^\ast,x_2^\ast\in S(x,\alpha)$, and
\[\alpha>\|x_1^\ast-x_2^\ast\|\geq x_1^\ast(y)-x_2^\ast(y)\geq 2-2\alpha.\]
Since $\alpha$ may be taken arbitrarily small, this leads to a contradiction.
\end{remark}

In the following Propositions~\ref{prop: reformulations of loc oct}--\ref{prop: reformulations of oct}, we point out some equivalent but sometimes more convenient formulations of octahedrality.

\begin{proposition}\label{prop: reformulations of loc oct}
Let $X$ be a Banach space.
The following assertions are equivalent:
\begin{itemize}
\item[{\rm(i)}]
$X$ is locally octahedral;
\item[(ii)] whenever $x\in S_X$ and $\eps>0$, there is a $y\in S_X$ such that
\begin{equation}\label{eq: norga oktaeedrilisuse tingimus t>0 abil}
\|x\pm ty\|\geq (1-\eps)\bigl(\|x\|+t\bigr)\quad\text{for all $t>0$};
\end{equation}
\item[{\rm(iii)}]
whenever $x\in S_X$ and $\eps>0$, there is a $y\in S_X$ such that
\begin{equation*}
\|x\pm y\|\geq 2-\eps.
\end{equation*}

\end{itemize}
\end{proposition}
\begin{proof}
(i)$\Leftrightarrow$(ii)$\Rightarrow$(iii) is obvious.

\medskip
(iii)$\Rightarrow$(ii).
Assume that (iii) holds. Let $x\in S_X$ and let $\eps>0$. By (iii), pick any $y\in S_X$ with $\|x\pm y\|\geq 2-\eps$. We show that $y$ satisfies (\ref{eq: norga oktaeedrilisuse tingimus t>0 abil}). Suppose that $t>0$. Then
\begin{align*}
\|x\pm ty\|&\geq\max\{1,t\}\|x\pm y\|-\bigr(\max\{1,t\}-\min\{1,t\}\bigl)\\
&\geq \max\{1,t\}(1-\eps)+\min\{1,t\}\\
&=1+t-\max\{1,t\}\eps\\
&\geq(1+t)(1-\eps).
\end{align*}
Thus $y$ satisfies (\ref{eq: norga oktaeedrilisuse tingimus t>0 abil}).
\end{proof}

\begin{proposition}\label{prop: reformulations of weakly oct}
Let $X$ be a Banach space.
The following assertions are equivalent:
\begin{itemize}
\item[{\rm(i)}]
$X$ is weakly octahedral;
\item[{\rm(ii)}]
whenever $E$ is a finite-dimensional subspace of $X$, $\xs\in B_{\Xs}$, and $\eps>\nobreak0$,
there is~a $y\in S_X$ such that
\[
\|x+ty\|\geq(1-\eps)\bigl(|\xs(x)|+t\bigr)\quad\text{for all $x\in S_E$ and $t>0$;}
\]
\item[{\rm(ii')}]
whenever $n\in\N$, $x_1,\dotsc,x_n\in S_X$, $\xs\in B_{\Xs}$, and $\eps>\nobreak0$, there is~a $y\in S_X$ such that
\[
\|x_i+ty\|\geq(1-\eps)\bigl(|\xs(x_i)|+t\bigr)\quad\text{for all $i\in\{1,\dotsc,n\}$ and $t>0$;}
\]
\item[{\rm(iii)}]
whenever $n\in\N$, $x_1,\dotsc,x_n\in S_X$, $\xs\in B_{\Xs}$, and $\eps>\nobreak0$, there is~a $y\in S_X$ such that
\[
\|x_i+ty\|\geq(1-\eps)\bigl(|\xs(x_i)|+t\bigr)\quad\text{for all $i\in\{1,\dotsc,n\}$ and $t\geq\eps$.}
\]
\end{itemize}
\end{proposition}
\begin{proof}
(i)$\Leftrightarrow$(ii)$\Rightarrow$(ii')$\Rightarrow$(iii) is obvious.

\medskip
(iii)$\Rightarrow$(ii).
Assume that (iii) holds.
Let $E$ be a nontrivial finite-dimensional subspace of $X$, let $\xs\in B_{\Xs}$, and let $0<\eps<1$.
Pick $\delta>0$ satisfying $\eps\geq(2-\eps)\delta$, and $\gamma>0$ satisfying $\gamma(2-\delta)\leq\delta^2$.
Let $A\subset S_E$ be a finite $\gamma$-net for $S_E$.
By (iii), there is a $y\in S_X$ satisfying
\[
\|z+ty\|\geq(1-\delta)\bigl(|\xs(z)|+t\bigr)\quad\text{for all $z\in A$ and all $t\geq\delta$.}
\]
Let $x\in S_E$ and $t>0$ be arbitrary.
First suppose that $t\leq\delta$.
In this case, observing that $-\delta\geq-\eps+\delta-\eps\delta$, i.e. $1-\delta\geq(1-\eps)(1+\delta)$,
and thus also $1-\delta\geq(1-\eps)(1+t)$,
\[
\|x+ty\|\geq1-\delta\geq(1-\eps)(1+t)\geq(1-\eps)\bigl(|\xs(x)|+t\bigr).
\]
Now consider the case $t\geq\delta$. Letting $z\in A$ be such that $\|x-z\|<\gamma$, one has
\begin{align*}
\|x+ty\|
&\geq\|z+ty\|-\gamma\geq(1-\delta)\bigl(|\xs(z)|+t\bigr)-\gamma\\
&\geq(1-\delta)\bigl(|\xs(x)|+t\bigr)-\gamma(1-\delta)-\gamma.
\end{align*}
Since $t\geq\delta$, one has
\[
\gamma(1-\delta)+\gamma=\gamma(2-\delta)\leq\delta^2\leq\delta\bigl(|\xs(x)|+t\bigr),
\]
and it follows that
\[
\|x+ty\|\geq(1-2\delta)\bigl(|\xs(x)|+t\bigr)\geq(1-2\eps)\bigl(|\xs(x)|+t\bigr).
\]
\end{proof}

\begin{proposition}\label{prop: reformulations of weakly oct for X*}
Let $X$ be a Banach space.
The following assertions are equivalent:
\begin{itemize}
\item[{\rm(i)}]
$X^\ast$ is weakly octahedral;
\item[{\rm(ii)}]
whenever $E$ is a finite-dimensional subspace of~$\Xs$, $x\in B_{X}$, and $\eps>\nobreak0$,
there is~a $\ys\in S_{\Xs}$ such that
\[
\|\xs+\ys\|\geq(1-\eps)\bigl(|\xs(x)|+\|\ys\|\bigr)\quad\text{for all $\xs\in E$;}
\]

\item[{\rm(iii)}]
whenever $n\in\N$, $\xs_1,\dotsc,\xs_n\in S_{\Xs}$, $x\in B_{X}$, and $\eps>\nobreak0$, there is~a $\ys\in S_{\Xs}$ such that
\[
\|\xs_i+t\ys\|\geq(1-\eps)\bigl(|\xs_i(x)|+t\bigr)\quad\text{for all $i\in\{1,\dotsc,n\}$ and $t\geq\eps$.}
\]
\end{itemize}
\end{proposition}
\begin{proof}
(i)$\Rightarrow$(ii)$\Rightarrow$(iii) is obvious.

\medskip
(iii)$\Rightarrow$(ii) is similar to (iii)$\Rightarrow$(ii) in the proof of Proposition~\ref{prop: reformulations of weakly oct}.

\medskip
(ii)$\Rightarrow$(i). This is a standard use of the principle of local reflexivity. Alternatively, one may use an appropriate $\eps$-net for $S_E$ and Goldstine's theorem.
\end{proof}
\begin{proposition}\label{prop: reformulations of oct}
Let $X$ be a Banach space.
The following assertions are equivalent:
\begin{itemize}
\item[{\rm(i)}]
$X$ is octahedral;
%
%
\item[{\rm(ii)}]
whenever $E$ is a finite-dimensional subspace of $X$ and $\eps>\nobreak0$,
there is~a $y\in S_X$ such that
\begin{equation}\label{eq: oktaeedrilisuse tingimus t>0 abil}
\|x+ty\|\geq(1-\eps)\bigl(\|x\|+t\bigr)\quad\text{for all $x\in S_E$ and $t>0$;}
\end{equation}
\item[{\rm(ii')}]
whenever $n\in\N$, $x_1,\dotsc,x_n\in S_X$, and $\eps>\nobreak0$, there is~a $y\in S_X$ such that
\begin{equation*}
\|x_i+ty\|\geq(1-\eps)\bigl(\|x_i\|+t\bigr)\quad\text{for all $i\in\{1,\dotsc,n\}$ and $t>0$;}
\end{equation*}
\item[{\rm(iii)}]
whenever $n\in\N$, $x_1,\dotsc,x_n\in S_X$, and $\eps>\nobreak0$, there is~a $y\in S_X$ such that
\[
\|x_i+y\|\geq2-\eps\quad\text{for all $i\in\{1,\dotsc,n\}$.}
\]
\end{itemize}
\end{proposition}
\begin{proof}
(i)$\Leftrightarrow$(ii)$\Rightarrow$(ii')$\Rightarrow$(iii) is obvious.

\medskip
(iii)$\Rightarrow$(ii') is similar to (iii)$\Rightarrow$(ii) in the proof of Proposition~\ref{prop: reformulations of loc oct}.

\medskip
(ii')$\Rightarrow$(ii). Assume that (ii') holds. Let $E$ be a nontrivial finite-dimensional subspace of $X$ and let $\eps>0$.
We shall show that there is a $y\in S_X$ satisfying (\ref{eq: oktaeedrilisuse tingimus t>0 abil}).
Let $A\subset S_X$ be a finite $\eps/2$-net for $S_E$. By (ii'), there is a $y\in S_X$ satisfying
\[
\|z+ty\|\geq(1-\frac{\eps}{2})(\|z\|+t)\quad\text{for all $z\in A$ and $t>0$}.
\]
Let $x\in S_X$ and $t>0$ be arbitrary. Letting $z\in A$ be such that $\|x-z\|<\eps/2$, one has
\begin{align*}
\|x+ty\|&\geq\|z+ty\|-\|x-z\|\\
&\geq(1-\dfrac{\eps}{2})(1+t)-\dfrac{\eps}{2}\\
&\geq(1-\eps)(1+t).
\end{align*}

\end{proof}
\section{Criteria for weak$^\ast$ diameter 2 properties}\label{sec: 3. Criteria for weak* d2p-s}
In this section, the duality between diameter $2$ properties and octahedrality is established. We also show that one may think of diameter $2$ properties as sort of extension properties.

%
\begin{theorem}\label{thm: omnibus thm for the weak* ld2P}
Let $X$ be a Banach space.
The following assertions are equivalent:
\begin{itemize}
\item[{\rm(i)}]
$\Xs$ has the weak$^{\ast}$ local diameter $2$ property;
\item[{\rm(ii)}]
$X$ is locally octahedral;
\item[{\rm(iii)}]
for every $x\in S_X$, every $\alpha\in[-1,1]$, every $\eps>\nobreak0$, and every $\eps_0\in(0,\eps)$,
there is a $y\in S_X$ such that, whenever $|\gamma|\leq1+\eps_0$, there is a $\ys\in\Xs$ satisfying
\[
\ys(x)=\alpha,\quad \ys(y)=\gamma,\quad\text{and}\quad \|\ys\|\leq1+\eps;
\]
\item[{\rm(iii')}]
for every $x\in S_X$, every $\alpha\in[-1,1]$, and every $\eps>\nobreak0$,
there are $y\in S_X$ and $\xs_1,\xs_2\in\Xs$ satisfying
\begin{equation*}
\xs_1(x)=\xs_2(x)=\alpha,\quad \xs_1(y)-\xs_2(y)>2-\eps,
\end{equation*}
and $\|\xs_1\|,\|\xs_2\|\leq 1+\eps$;
\item[{\rm(iii'')}]
for every $x\in S_X$ and every $\eps>\nobreak0$,
there are $y\in S_X$ and $\xs_1,\xs_2\in\Xs$ satisfying
\begin{equation*}
\xs_1(x)=\xs_2(x)=1,\quad \xs_1(y)-\xs_2(y)>2-\eps,
\end{equation*}
 and $\|\xs_1\|,\|\xs_2\|\leq 1+\eps$.
\end{itemize}
\end{theorem}
\begin{proof}
(i)$\Rightarrow$(ii).
Assume that (i) holds. Let $x\in S_X$ and let $\eps>0$.
By~(i), there are $x_1^\ast,x_2^\ast\in B_{X^\ast}$ and $y\in S_X$ such that
\[x_1^\ast(x),x_2^\ast(x)>1-\eps\quad \text{and}\quad x_1^\ast(y)-x_2^\ast(y)>2-\eps.\]
By (the equivalence (i)$\Leftrightarrow$(iii) of) Proposition \ref{prop: reformulations of loc oct}, it suffices to show that
\[\|x\pm y\|\geq 2-2\eps.\]
Since $x_1^\ast(x),x_2^\ast(x)> 1-\eps$ and $x_1^\ast(y),-x_2^\ast(y)>1-\eps$, it follows that
\[
\|x+y\|\geq\xs_1(x+y)>2-2\eps
\]
and
\[
\|x-y\|\geq\xs_2(x-y)>2-2\eps.
\]

\medskip
(ii)$\Rightarrow$(iii).
Assume that (ii) holds.
Let $x\in S_X$, let $\alpha\in[-1,1]$, and let $0<\eps_0<\eps$.
Choose $y\in S_X$ to satisfy
\[
\|sx+y\|\geq\frac{1+\eps_0}{1+\eps}\bigl(|s|+\|y\|\bigr)\quad\text{for all $s\in\R$.}
\]
Now let $|\gamma|\leq1+\eps_0$.
Defining $g\in\bigl(\spann\{x,y\}\bigr)^\ast$ by
\[
g(x)=\alpha,\quad g(y)=\gamma,
\]
one has, for all $s\in\R$,
\[
\bigl|g(sx+y)\bigr|\leq|s||\alpha|+|\gamma|\leq(1+\eps_0)\bigl(|s|+\|y\|\bigr)\leq(1+\eps)\|sx+y\|,
\]
hence $\|g\|\leq 1+\eps$. The desired $\ys$ can be defined to be any norm-preserving extension to $X$ of $g$.

\medskip
(iii)$\Rightarrow$(iii')$\Rightarrow$(iii'') is obvious.

\medskip
(iii'')$\Rightarrow$(i).
Let $x\in S_X$ and $\eps>0$ be arbitrary, and let $y\in S_X$ and $\xs_1,\xs_2\in\Xs$ be as in (iii'').
It suffices to observe that $\frac{\xs_1}{1+\eps},\frac{\xs_2}{1+\eps}\in B_{\Xs}$,
\[
\biggl\|\frac{\xs_1}{1+\eps}-\frac{\xs_2}{1+\eps}\biggr\|>\frac{\bigl|\xs_1(y)-\xs_2(y)\bigr|}{1+\eps}>\frac{2-\eps}{1+\eps},
\]
and, for all $i\in\{1,2\}$,
\[
\frac{\xs_i}{1+\eps}(x)=\frac1{1+\eps}.
\]
\end{proof}
The following theorem is an obvious consequence of Proposition~\ref{prop: X d2p = X** weak* d2p} and Theorem~\ref{thm: omnibus thm for the weak* ld2P}.
\begin{theorem}\label{thm: omnibus thm for the ld2P}
Let $X$ be a Banach space.
The following assertions are equivalent:
\begin{itemize}
\item[{\rm(i)}]
$X$ has the local diameter $2$ property;
\item[{\rm(ii)}]
$X^\ast$ is locally octahedral.
\end{itemize}
\end{theorem}
%
%
\begin{theorem}\label{thm: omnibus thm for the weak* d2P}
Let $X$ be a Banach space.
The following assertions are equivalent:
\begin{itemize}
\item[{\rm(i)}]
$\Xs$ has the weak$^{\ast}$ diameter $2$ property;
\item[{\rm(ii)}]
$X$ is weakly octahedral;
\item[{\rm(iii)}]
for every finite-dimensional subspace $E$ of $X$, every $\xs\in B_{\Xs}$, every $\eps>\nobreak0$, and every $\eps_0\in(0,\eps)$,
there is a $y\in S_X$ such that, whenever $|\gamma|\leq1+\eps_0$, there is a $\ys\in\Xs$ satisfying
\[
\ys|_E=\xs|_E,\quad \ys(y)=\gamma,\quad\text{and}\quad \|\ys\|\leq1+\eps.
\]
\item[{\rm(iii')}]
for every finite-dimensional subspace $E$ of $X$, every $\xs\in B_{\Xs}$, and every $\eps>\nobreak0$,
there are $y\in S_X$ and $\xs_1,\xs_2\in\Xs$ satisfying
\begin{equation*}\label{eq: jatkude tingimus weak* d2P kirjelduses}
\xs_1|_E=\xs_2|_E=\xs|_E,\quad \xs_1(y)-\xs_2(y)>2-\eps,
\end{equation*}
and $\|\xs_1\|,\|\xs_2\|\leq 1+\eps$;
\end{itemize}
\end{theorem}
\begin{proof}
(i)$\Rightarrow$(ii).
Assume that (i) holds. Let $x_1,\dotsc,x_n\in S_X$ ($n\in\N$), let $\xs\in B_{\Xs}$, and let $0<\eps<1$.
Pick $\delta\in(0,\eps^2)$ satisfying $\delta<\eps\,|\xs(x_i)|$ for all $i\in\{1,\dotsc,n\}$ with $|\xs(x_i)|\not=0$.
By (i), there are $\us,\vs\in B_{\Xs}$ and $y\in S_X$ such that
\[
\bigl|\us(x_i)-\xs(x_i)\bigr|<\delta
\quad
\text{and}\quad
\bigl|\vs(x_i)-\xs(x_i)\bigr|<\delta
\quad\text{for all $i\in\{1,\dotsc,n\}$,}
\]
and
\[
\vs(y)-\us(y)>2-\eps.
\]
Since $\vs(y)\leq1$ and $\us(y)\geq-1$, it follows that $\vs(y)>1-\eps$ and $\us(y)<-1+\eps$.
Let $i\in\{1,\dotsc,n\}$ and $t\geq\eps$ be arbitrary.
If $\xs(x_i)\not=0$, then, choosing $\zs\in\{\us,\vs\}$ so that $\xs(x_i)$ and $\zs(y)$
(and thus also $\zs(x_i)$ and $\zs(y)$) have the same sign, one has 
\begin{align*}
\|x_i+ty\|
&\geq\bigl|\zs(x_i)+t\zs(y)\bigr|
=|\zs(x_i)|+t|\zs(y)|\\
&\geq|\xs(x_i)|-|\xs(x_i)-\zs(x_i)|+t|\zs(y)|\\
&\geq|\xs(x_i)|-\eps|\xs(x_i)|+(1-\eps)t\\
&=(1-\eps)\bigl(|\xs(x_i)|+t\bigr).
\end{align*}
If $\xs(x_i)=0$, then
\begin{align*}
\|x_i+ty\|
&\geq\bigl|\us(x_i)+t\us(y)\bigr|\geq t|\us(y)|-|\us(x_i)|\\
&\geq(1-\eps)t-\eps^2
\geq(1-\eps)t-t\eps=(1-2\eps)\bigl(|\xs(x_i)|+t\bigr),
\end{align*}
and it follows that $X$ is weakly octahedral.

\medskip
(ii)$\Rightarrow$(iii).
Assume that (ii) holds.
Let $E$ be a finite-dimensional subspace of $X$, let $\xs\in B_{\Xs}$, and let $0<\eps_0<\eps$.
Choose $y\in S_X$ to satisfy
\[
|\xs(x)|+|t|\leq\frac{1+\eps}{1+\eps_0}\|x+ty\|\quad\text{for all $x\in E$ and all $t\in\R$.}
\]
Letting $\gamma\in[-1-\eps_0,1+\eps_0]$, and defining $g\in\bigl(\spann(E\cup\{y\})\bigr)^\ast$ by $g|_E=\xs|_E$ and $g(y)=\gamma$,
it suffices to show that $\|g\|\leq1+\eps$ (because, in this case, one may define the desired $\ys\in\Xs$ to be any norm-preserving extension of $g$).
To this end, it remains to observe that, whenever $x\in E$ and $t\in\R$,
\begin{align*}
|g(x+ty)|
&\leq|\xs(x)|+|t|\,|\gamma|\leq(1+\eps_0)(|\xs(x)|+|t|)\\
&\leq(1+\eps)\|x+ty\|.
\end{align*}

\medskip
(iii)$\Rightarrow$(iii') is obvious.

\medskip
(iii')$\Rightarrow$(i).
Let $\xs\in B_{\Xs}$, let $x_1,\dotsc,x_n\in S_X$ ($n\in\N$), and let $\eps>0$.
Put $E:=\spann\{x_1,\dotsc,x_n\}$, and let $y\in S_X$ and $\xs_1,\xs_2\in\Xs$ be as in (iii').
It suffices to observe that $\frac{\xs_1}{1+\eps},\frac{\xs_2}{1+\eps}\in B_{\Xs}$,
\[
\biggl\|\frac{\xs_1}{1+\eps}-\frac{\xs_2}{1+\eps}\biggr\|\geq\frac{|\xs_1(y)-\xs_2(y)|}{1+\eps}>\frac{2-\eps}{1+\eps},
\]
and, for all $i\in\{1,\dotsc,n\}$ and $j\in\{1,2\}$,
\[
\biggl|\xs(x_i)-\frac{\xs_j}{1+\eps}(x_i)\biggr|=\biggl|\xs(x_i)-\frac{\xs(x_i)}{1+\eps}\biggr|=\frac\eps{1+\eps}|\xs(x_i)|<\eps.
\]
\end{proof}
The following theorem is an obvious consequence of Proposition~\ref{prop: X d2p = X** weak* d2p} and Theorem~\ref{thm: omnibus thm for the weak* d2P}.
\begin{theorem}\label{thm: omnibus thm for the d2P}
Let $X$ be a Banach space.
The following assertions are equivalent:
\begin{itemize}
\item[{\rm(i)}]
$X$ has the diameter $2$ property;
\item[{\rm(ii)}]
$X^\ast$ is weakly octahedral.
\end{itemize}
\end{theorem}
%
%
\begin{theorem}\label{thm: omnibus thm for the weak* Sd2P}
Let $X$ be a Banach space.
The following assertions are equivalent:
\begin{itemize}
\item[{\rm(i)}]
$\Xs$ has the weak$^{\ast}$ strong diameter $2$ property;
\item[{\rm(ii)}]
$X$ is octahedral;
\item[{\rm(iii)}]
whenever $E$ is a finite-dimensional subspace of $X$, $n\in\N$, $\xs_1,\dotsc,\xs_n\in B_{\Xs}$, $\eps>\nobreak0$, and $\eps_0\in(0,\eps)$,
there is a $y\in S_X$ such that, whenever $|\gamma_i|\leq1+\eps_0$, $i\in\{1,\dotsc,n\}$, there are $\ys_i\in\Xs$ satisfying
\[
\ys_i|_E=\xs_i|_E,\quad \ys_i(y)=\gamma_j,\quad\text{and}\quad \|\ys_i\|\leq1+\eps\quad\text{for all $i\in\{1,\dotsc,n\}$;}
\]
\item[{\rm(iii')}]
whenever $E$ is a finite-dimensional subspace of $X$, $n\in\N$, $\xs_1,\dotsc,\xs_n\in B_{\Xs}$, and $\eps>\nobreak0$,
there are $y\in S_X$ and $\xs_{1i},\xs_{2i}\in\nobreak\Xs$, $i\in\{1,\dotsc,n\}$, satisfying
\begin{equation*}\label{eq: jatkude tingimus weak* Sd2P kirjelduses}
\xs_{1i}|_E=\xs_{2i}|_E=\xs_i|_E,\quad \xs_{1i}(y)-\xs_{2i}(y)>2-\eps,
\end{equation*}
and $\|\xs_{1i}\|,\|\xs_{2i}\|\leq 1+\eps$ for all $i\in\{1,\dotsc,n\}$.
\end{itemize}
\end{theorem}
\begin{proof}
The equivalence (i)$\Leftrightarrow$(ii) was pointed out in \cite[p.~12]{G}. Since no details of the proof were given in \cite{G}, we include the proof for completeness.

(i)$\Rightarrow$(ii). Assume that (i) holds. Let $x_1,\dotsc,x_n\in S_X$ ($n\in\mathbb N$) and let $\eps>0$. By (i), for every $i\in\{1,\dotsc,n\}$, there are $x_{1i}^\ast,x_{2i}^\ast\in B_{\Xs}$ and $y\in S_X$ such that
\[
\xs_{1i}(x_i),\xs_{2i}(x_i)>1-\eps\quad\text{and}\quad\frac1n\sum_{i=1}^n\big(x_{1i}^\ast(y)-x_{2i}^\ast(y)\big)>2-\frac{\eps}{n}.
\]
For every $i\in\{1,\dotsc,n\}$, since $x_{1i}^\ast(y)>1-\eps$, one has
\[
\|x_i+y\|\geq x_{1i}^\ast(x_i+y)>2-2\eps,
\]
and $X$ is octahedral by (the equivalence (i)$\Leftrightarrow$(iii) of) Proposition~\ref{prop: reformulations of oct}.

\medskip
(ii)$\Rightarrow$(iii). Assume that (ii) holds.
Let $E\subset X$ be a finite-dimensional subspace, let $n\in\N$, let $\xs_1,\dots,\xs_n\in B_\Xs$, and let $0<\eps_0<\eps$.
Choose $y\in S_X$ to satisfy
\[
\|x+ty\|\geq\frac{1+\eps_0}{1+\eps}\bigl(\|x\|+t\bigr)\quad\text{for all $x\in S_E$ and $t>0$.}
\]
Now let $|\gamma_i|\leq1+\eps_0$, $i\in\{1,\dots,n\}$.
For every $i\in\{1,\dotsc,n\}$, defining $g_i\in\Bigl(\spann \bigl(E\cup\{y\}\bigl)\Bigr)^\ast$ by
\[
g_i|_E=\xs_i|_E,\quad g_i(y)=\gamma_i,
\]
one has, for all $x\in S_E$ and $t>0$,
\[
\bigl|g_i(x+ty)\bigr|\leq|\xs_i(x)|+t|\gamma_i|\leq(1+\eps_0)\bigl(|\xs_i(x)|+t\bigr)\leq(1+\eps)\|x+ty\|,
\]
hence $\|g_i\|\leq 1+\eps$. The desired $\ys_1,\dotsc,\ys_n$ can be defined to be any norm-preserving extension to $X$
of $g_1,\dotsc,g_n$, respectively.

\medskip
(iii)$\Rightarrow$(iii') is obvious.

\medskip
(iii')$\Rightarrow$(i).
Let $n\in\N$, let $x_1,\dots, x_n\in S_X$, let $\lambda_1,\dotsc,\lambda_n\geq 0$, $\sum_{i=1}^{n}\lambda_i=1$,
and let $\eps>0$.
For every $i\in\{1,\dotsc,n\}$, choose $\xs_i\in B_{\Xs}$, so that $\xs_i(x_i)>1-\eps$,
and let $y\in S_X$ and $\xs_{11},\xs_{21},\dotsc,\xs_{1n},\xs_{2n}\in\nobreak\Xs$ as in (iii'), where $E=\spann\{x_1,\dotsc,x_n\}$.
It suffices to observe that, for all $i\in\{1,\dotsc,n\}$ and $j\in\{1,2\}$,
one has $\frac{\xs_{ji}}{1+\eps}\in B_{\Xs}$,
\[
\frac{\xs_{ji}}{1+\eps}(x_i)=\frac{\xs_i(x_i)}{1+\eps}>\frac{1-\eps}{1+\eps},
\]
and
\[
\biggl\|\sum_{i=1}^{n}\lambda_i\frac{\xs_{1i}}{1+\eps}-\sum_{i=1}^{n}\lambda_i\frac{\xs_{2i}}{1+\eps}\biggr\|
>\frac{\Bigl|\sum_{j=1}^{n}\lambda_j\bigl(\xs_{1i}(y)-\xs_{2i}(y)\Bigr)\Bigr|}{1+\eps}>\frac{2-\eps}{1+\eps}.
\]
\end{proof}
The following theorem is an obvious consequence of Proposition~\ref{prop: X d2p = X** weak* d2p} and Theorem~\ref{thm: omnibus thm for the weak* Sd2P}.
\begin{theorem}\label{thm: omnibus thm for the sd2P}
Let $X$ be a Banach space.
The following assertions are equivalent:
\begin{itemize}
\item[{\rm(i)}]
$X$ has the strong diameter $2$ property;
\item[{\rm(ii)}]
$\Xs$ is octahedral.
\end{itemize}
\end{theorem}
%

In \cite[Theorem III.2.5]{DGZ}, it was shown that a Banach space has an equivalent octahedral norm if and only if it contains an isomorphic copy of $\ell_1$.

\begin{corollary}
If a Banach space $X$ has the strong diameter $2$ property, then $X^\ast$ contains a subspace isomorphic to $\ell_1$.
\end{corollary}

\section{Stability results}\label{sec: 4. Stability of octahedrality}

We begin by recalling that the (local) diameter $2$ property is stable by taking $\ell_p$-sums not only if $p=1$ and $p=\infty$, but surprisingly for all $1\leq p\leq\infty$ (see \cite{ALN}, \cite{P}, and \cite{GP}).
(Some further development was carried out in \cite{ABL} where, instead of $p$-sums, product spaces with absolute norm were considered.)
If $1<p<\infty$, then $p$-sums of Banach spaces lack the strong diameter $2$ property (see \cite{ABL}; see also \cite{HL}).
However, if $p=1$ or $p=\infty$, then the $p$-sum may have the strong diameter $2$ property (see e.g. \cite[Theorem 2.7, (iii), and Proposition 4.6]{ALN} and \cite[Proposition 3.1]{ABL}).


The following proposition is our main stability result for locally octahedral spaces.

\begin{proposition}\label{prop: p-sum LOH}
Let $X$ and $Y$ be Banach spaces.
\begin{itemize}
\item[{\rm(a)}]
If $X$ is locally octahedral, then $X\oplus_1 Y$ is locally octahedral.
\item[{\rm(b)}]
If $X$ and $Y$ are locally octahedral, and $1<p\leq\infty$, then $X\oplus_p Y$ is locally octahedral.
\item[{\rm(c)}]
If $X\oplus_p Y$ is locally octahedral, where $1<p\leq\infty$, then $X$ is locally octahedral.
\end{itemize}
\end{proposition}
\begin{remark}
Note that Proposition \ref{prop: p-sum LOH}, (c), fails if we take $p=1$. Indeed, by part (a) of Proposition \ref{prop: p-sum LOH}, $\ell_1\oplus_1 \R$ is locally octahedral, but $\R$ fails to be locally octahedral.
\end{remark}
%

\begin{proof}
(a).
Assume that $X$ is locally octahedral. Fix $(x,y)\in S_{X\oplus_1 Y}$ and $\eps>0$. By our assumption, there exists a $u\in S_X$ such that
\[
\|x\pm u\|\geq(1-\eps)\bigl(\|x\|+1\bigr).
\]
Hence,
\[
\bigl\|(x,y)\pm(u,0)\bigr\|_1\geq(1-\eps)\bigl(\|x\|+1\bigr)+\|y\|\geq 2-2\eps.
\]
Thus $X\oplus_1 Y$ is locally octahedral.

\medskip
(b).
Assume that $X$ and $Y$ are locally octahedral, and let $1<p\leq\infty$.
Let $(x,y)\in S_{X\oplus_p Y}$ and let $0<\eps<1$.
It suffices to find a $(u,v)\in S_{X\oplus_p Y}$ such that
\[
\|(x,y)\pm (u,v)\|_p\geq 2-2\eps.
\]

We may (and do) assume that $x\not=0$ and $y\not=0$.
By our assumption, there exist $\tilde u\in S_X$ and $\tilde v\in S_Y$ such that
\[
\|x+t\tilde u\|\geq(1-\eps)\bigl(\|x\|+|t|\bigr)\qquad\text{for all $t\in\mathbb R$}
\]
and
\[
\|y+t\tilde v\|\geq(1-\eps)\bigl(\|y\|+|t|\bigr)\qquad\text{for all $t\in\mathbb R$}.
\]

If $1<p<\infty$, it follows that
\[
\Big\|x\pm\|x\|\tilde{u}\Big\|^p+\Big\|y\pm\|y\|\tilde{v}\Big\|^p\geq(1-\eps)^p\, 2^p.
\]
This completes the proof for $1<p<\infty$, because one may take $u=\|x\|\tilde u$ and $v=\|y\|\tilde v$.

If $p=\infty$, one may take $u=\tilde u$ and $v=\tilde v$  because
\begin{align*}
\bigl\|(x,y)\pm (\tilde u,\tilde v)\bigr\|_\infty&=\max\bigl\{\|x\pm \tilde u\|,\|y\pm \tilde v\|\bigr\}\\
&\geq (1-\eps)\bigl(\max\{\|x\|,\|y\|\}+1\bigr)\\
&=2-2\eps.
\end{align*}

(c).
Assume that $X\oplus_p Y$ is locally octahedral, where $1<p\leq\infty$.
Let $x\in S_X$ and let $0<\eps<1$. Since $\bigl\|(x,0)\bigl\|_p=1$, whenever $\delta>0$,
there exists a $(u,v)\in S_{X\oplus_p Y}$ such that
\begin{equation}\label{eq: ||(x+-u,v)||_p>=2-delta}
\bigl\|(x\pm u,v)\bigr\|_p\geq2-\delta.
\end{equation}
It suffices to show that (\ref{eq: ||(x+-u,v)||_p>=2-delta}) with $\delta$ small enough implies that
\begin{equation}\label{eq: ||x+-u||>=2-eps}
\|x\pm u\|\geq 2-\eps,
\end{equation}
because, in this case,  $\|u\|\geq 1-\eps$, thus
\[
\biggl\|x\pm \frac{u}{\|u\|}\biggr\|\geq\|x\pm u\|-\bigl(1-\|u\|\bigr)\geq 2-2\eps,
\]
and it follows that $X$ is locally octahedral.

If $p=\infty$, (\ref{eq: ||(x+-u,v)||_p>=2-delta}) means that $\max\bigl\{\|x\pm u\|,\|v\|\bigr\}\geq 2-\delta$.
Since $\|v\|\leq 1$, taking $\delta=\eps$ implies~(\ref{eq: ||x+-u||>=2-eps}).

If $1<p<\infty$, (\ref{eq: ||(x+-u,v)||_p>=2-delta}) means that $\|x\pm u\|^p+\|v\|^p\geq (2-\delta)^p$.
Since $\|u\|^p+\|v\|^p=1$, this implies that
\begin{equation}\label{eq: ||x+-u||>=something}
\|x\pm u\|^p\geq(2-\delta)^p-\bigl(1-\|u\|^p\bigr),
\end{equation}
thus it suffices to show that $\|u\|$ is as close to $1$ as we want whenever $\delta$ is small enough.
The latter is true because, by (\ref{eq: ||x+-u||>=something}),
\[
\left(1+ \|u\|\right)^p-\|u\|^p\geq(2-\delta)^p-1,
\]
and the function $f\colon\,[0,1]\to\mathbb R$, $f(t)=(1+t)^p-t^p$, is strictly increasing with $\lim_{t\rightarrow 1}f(t)=2^p-1$.
\end{proof}

Proposition~\ref{prop: p-sum LOH} combined, respectively, with Theorems~\ref{thm: omnibus thm for the ld2P} and \ref{thm: omnibus thm for the weak* ld2P} immediately gives the corresponding stability results for the local diameter $2$ property and for the weak$^\ast$ local diameter $2$ property. These results for the local diameter $2$ property are known, but our method provides an alternative approach.

\begin{corollary}
Let $X$ and $Y$ be Banach spaces.
\begin{itemize}
\item[{\rm(a)}]
If $X$ has the local diameter~$2$ property, then $X\oplus_\infty Y$ has the local diameter $2$ property \emph{(cf. \cite[Theorem~3.2]{ALN}, see also \cite[Proposition~2.4]{ABL})}.
\item[{\rm(b)}]
If $X$ and $Y$ have the local diameter $2$ property, and $1\leq p<\infty$, then $X\oplus_p Y$ has the local diameter~$2$ property  \emph{(see \cite[Theorem~3.2]{ALN}, see also \cite[Proposition~2.4]{ABL})}.
\item[{\rm(c)}]
If $X\oplus_p Y$ has the local diameter $2$ property, where $1\leq p<\infty$, then $X$ has the local diameter $2$ property \emph{(see \cite[Proposition~2.5]{ABL})}.
\end{itemize}
\end{corollary}

\begin{corollary}
Let $X$ and $Y$ be Banach spaces.
\begin{itemize}
\item[{\rm(a)}]
If $X^\ast$ has the weak$^\ast$ local diameter $2$ property, then $(X\oplus_1 Y)^\ast$ has the weak$^\ast$ local diameter $2$ property.
\item[{\rm(b)}]
If $X^\ast$ and $Y^\ast$ have the weak$^\ast$ local diameter $2$ property, and $1< p\leq\infty$,
then $(X\oplus_p Y)^\ast$ has the weak$^\ast$ local diameter $2$ property.
\item[{\rm(c)}]
If $(X\oplus_p Y)^\ast$ has the weak$^\ast$ local diameter $2$ property, where $1< p\leq\infty$, then $\Xs$ has the weak$^\ast$ local diameter $2$ property.
\end{itemize}
\end{corollary}

The following proposition is our main stability result for weakly octahedral spaces.

%
\begin{proposition}\label{prop: p-sum WOH}
Let $X$ and $Y$ be Banach spaces.
\begin{itemize}
\item[{\rm(a)}]
If $X$ is weakly octahedral, then $X\oplus_1 Y$ is weakly octahedral.
\item[{\rm(b)}]
If $X$ and $Y$ are weakly octahedral, and $1<p\leq\infty$, then $X\oplus_p Y$ is weakly octahedral.
\item[{\rm(c)}]
If $X\oplus_p Y$ is weakly octahedral, where $1<p\leq\infty$, then $X$ is weakly octahedral.
%
\end{itemize}
\end{proposition}

\begin{proof}
(a).
Assume that $X$ is weakly octahedral.
Let $E$ and $F$ be finite-dimensional subspaces of $X$ and $Y$, respectively,
let $(x^\ast,y^\ast)\in B_{X^\ast\oplus_\infty Y^\ast}$, and let $\eps>0$.
It suffices to show that there exists a $(u,v)\in S_{X\oplus_1 Y}$ such that,  for all $x\in E$, all $y\in F$, and all $t\in\mathbb R$,
one has
\[
\bigl\|(x,y)+t(u,v)\bigr\|_1
\geq (1-\eps)\bigl(|x^\ast(x)+y^\ast(y)|+|t|\bigr).
\]
By our assumption, there exists a $u\in S_X$ such that
\[
\|x+tu\|\geq (1-\eps)\bigl(|x^\ast(x)|+|t|\bigr)
\quad\text{for all $x\in E$ and all $t\in\mathbb R$,}
\]
One has, for all $x\in E$, all $y\in F$, and all $t\in\mathbb R$,
\begin{align*}
\bigl\|(x,y)+t(u,0)\bigr\|
&=\|x+tu\|+\|y\|\\
&\geq (1-\eps)\bigl(|x^\ast(x)|+|t|\bigr)+\|y\|\\
&\geq (1-\eps)\bigl(|x^\ast(x)+\ys(y)|+|t|\bigr).
\end{align*}
Thus $X\oplus_1 Y$ is weakly octahedral.

\medskip
(b).
Assume that $X$ and $Y$ are weakly octahedral, and let $1<p\leq \infty$.
Let $E$ and $F$ be finite-dimensional subspaces of $X$ and $Y$, respectively,
let $(x^\ast,y^\ast)\in S_{X^\ast\oplus_q Y^\ast}$, where $q$ is the conjugate exponent of $p$
(i.e., $1/p+1/q=1$ if $1<p<\infty$, and $q=1$ if $p=\infty$), and let $0<\eps<1$.
It suffices to find a $(u,v)\in S_{X\oplus_p Y}$ such that for all $x\in E$, all $y\in F$, and all $t\in\mathbb R$, one has
\[
\bigl\|(x,y)+t\,(u,v)\bigr\|_p\geq (1-\eps)\bigl(|x^\ast(x)+y^\ast(y)|+|t|\bigr).
\]
We may (and do) assume that $x^\ast\not=0$ and $y^\ast\not=0$.

By our assumption, there exist $\tilde u\in S_X$ and $\tilde v\in S_Y$ such that
\[
\|x+t\tilde u\|\geq (1-\eps)\biggl(\frac{|x^\ast(x)|}{\|x^\ast\|}+|t|\biggr)
\quad\text{for all $x\in E$ and all $t\in\mathbb R$,}
\]
and
\[
\|y+t\tilde v\|\geq (1-\eps)\biggl(\frac{|y^\ast(y)|}{\|y^\ast\|}+|t|\biggr)
\quad\text{for all $y\in F$ and all $t\in\mathbb R$.}
\]

If $1<p<\infty$, take $u=\|x^\ast\|^{q-1}\tilde u$ and $v=\|y^\ast\|^{q-1}\tilde v$,
and observe that, for all $x\in E$, all $y\in F$, and all $t\in\mathbb R$,
\begin{align*}
\frac1{(1-\eps)^p}\big(\|x&+tu\|^p+\|y+tv\|^p\big)\\
&\geq\biggl(\frac{|x^\ast(x)|}{\|x^\ast\|}+\|x^\ast\|^{q-1}|t|\biggr)^p+\biggl(\frac{|y^\ast(y)|}{\|y^\ast\|}+\|y^\ast\|^{q-1}|t|\biggr)^p\\
&=\|x^\ast\|^q\biggl(\frac{|x^\ast(x)|}{\|x^\ast\|^q}+|t|\biggr)^p+\|y^\ast\|^q\biggl(\frac{|y^\ast(y)|}{\|y^\ast\|^q}+|t|\biggr)^p\\
&\geq\biggl(\|x^\ast\|^q\frac{|x^\ast(x)|}{\|x^\ast\|^q}+\|y^\ast\|^q\frac{|y^\ast(y)|}{\|y^\ast\|^q}+|t|\biggr)^p\\
&=\big(|x^\ast(x)|+|y^\ast(y)|+|t|\big)^p.
\end{align*}
The last inequality holds because the function $[0,\infty)\to\mathbb R$, $s\mapsto(s+|t|)^p$ is convex for any fixed $t\in\mathbb R$.

If $p=\infty$, take $u=\tilde u$ and $v=\tilde v$, and observe that, for all $x\in E$, all $y\in F$, and all $t\in\mathbb R$,
\begin{align*}
\frac1{(1-\eps)}\max\bigl\{\|x&+tu\|,\|y+tv\|\bigr\}\\
&\geq\frac1{(1-\eps)}\Bigl(\|x^\ast\|\,\|x+tu\|+\|y^\ast\|\,\|y+tv\|\Bigr)\\
&\geq\|x^\ast\|\biggl(\frac{|x^\ast(x)|}{\|x^\ast\|}+|t|\biggr)+\|y^\ast\|\biggl(\frac{|y^\ast(y)|}{\|y^\ast\|}+|t|\biggr)\\
&=|x^\ast(x)|+|y^\ast(y)|+|t|.
\end{align*}

(c).
Assume that $X\oplus_p Y$ is weakly octahedral, where $1<p\leq\infty$.
Let $E$ be a finite-dimensional subspace of $X$, let $\xs\in S_{\Xs}$, and let $0<\eps<1$.
Choose $\delta>0$ to satisfy $(1+\delta)^q-(1-\delta)^q<\eps^q$, where $q$ is the conjugate exponent of $p$.
By enlarging $E$ if necessary, we may assume that $\|\xs|_E\|\geq1-\delta$
(notice that $X$ must be infinite-dimensional by Proposition \ref{prop: p-sum LOH}, (c), and Remark~\ref{rem: loc-oct is inf-dimensional}).
By (the equivalence (ii)$\Leftrightarrow$(iii) of) Theorem \ref{thm: omnibus thm for the weak* d2P},
there are $z\in X$, $y\in Y$, with $\|(y,z)\|_p=1$, and $\zs_i\in\Xs$, $\ys_i\in\Ys$, with $\|(\zs_i,\ys_i)\|_q\leq1+\delta$, $i=1,2$,
satisfying
\[
\zs_i|_E=\xs|_E
\quad\text{and}\quad
\zs_i(z)+\ys_i(y)=(-1)^i,
\qquad
i=1,2.
\]
Since
\[
\|\ys_i\|^q\leq(1+\delta)^q-\|\zs_i\|^q\leq(1+\delta)^q-(1-\delta)^q<\eps^q,\quad i=1,2,
\]
one has $|\ys_i(y)|<\eps$, $i=1,2$, and thus $\zs_2(z)>1-\eps$ and $\zs_1(z)<-1+\eps$.
Now let $x\in E$ be arbitrary. Choosing $i\in\{1,2\}$ so that $\xs(x)$ and $\zs_i(z)$ have the same sign,
one has
\begin{align*}
\left\|x+\frac{z}{\|z\|}\right\|
&\geq\frac1{1+\eps}\left|\zs_i(x)+\frac{\zs_i(z)}{\|z\|}\right|
=\frac1{1+\eps}\left(|\xs(x)|+\frac{|\zs_i(z)|}{\|z\|}\right)\\
&\geq\frac1{1+\eps}\bigl(|\xs(x)|+1-\eps\bigr)
\geq\frac{1-\eps}{1+\eps}\bigl(|\xs(x)|+1\bigr),
\end{align*}
and it follows that $X$ is weakly octahedral.

\end{proof}
Proposition~\ref{prop: p-sum WOH} combined, respectively, with Theorems~\ref{thm: omnibus thm for the d2P} and \ref{thm: omnibus thm for the weak* d2P} immediately gives the corresponding stability results for the diameter $2$ property and for the weak$^\ast$ diameter $2$ property.
These results for the diameter $2$ property are known, but our method provides an alternative approach.

\begin{corollary}
Let $X$ and $Y$ be Banach spaces.
\begin{itemize}
\item[{\rm(a)}]
If $X$ has the diameter $2$ property, then $X\oplus_\infty Y$ has the diameter~$2$ property \emph{(see \cite[Lemma~2.1]{P}, see also \cite[Theorem~2.7,~(ii), and Theorem~3.2]{ALN}, see also \cite[Theorem~2.4]{ABL})}.
\item[{\rm(b)}]
If $X$ and $Y$ have the diameter $2$ property, and $1\leq p<\infty$, then $X\oplus_p Y$ has the diameter $2$ property \emph{(see \cite[Theorem~3.2]{ALN}, see also \cite[Theorem~2.4]{ABL})}.
\item[{\rm(c)}]
If $X\oplus_p Y$ has the diameter $2$ property, and $1\leq p<\infty$, then $X$ has the diameter $2$ property \emph{(see \cite[Proposition~2.5]{ABL})}.
\end{itemize}
\end{corollary}

\begin{corollary}
Let $X$ and $Y$ be Banach spaces.
\begin{itemize}
\item[{\rm(a)}]
If $X^\ast$ has the weak$^\ast$ diameter $2$ property, then $(X\oplus_1 Y)^\ast$ has the weak$^\ast$ diameter $2$ property.
\item[{\rm(b)}] If $X^\ast$ and $Y^\ast$ have the weak$^\ast$ diameter $2$ property, and  $1< p\leq\infty$,
then $(X\oplus_p Y)^\ast$ has the weak$^\ast$ diameter $2$ property.
\item[{\rm(c)}]
If $(X\oplus_p Y)^\ast$ has the weak$^\ast$ diameter $2$ property, and $1< p\leq\infty$, then $\Xs$ has the weak$^\ast$ diameter $2$ property.
\end{itemize}
\end{corollary}

The following proposition is our main stability result for octahedral spaces.
It turns out that octahedral spaces are stable under $\ell_p$-sums only if $p=1$ or $p=\infty$.

\begin{proposition}\label{prop: p-sum OH}
Let $X$ and $Y$ be Banach spaces. \begin{itemize}
\item[{\rm(a)}] If $X$ is octahedral, then $X\oplus_1 Y$ is octahedral.
\item[{\rm(b)}] If $1<p<\infty$, then $X\oplus_p Y$ is not octahedral.
\item[{\rm(c)}] If $X$ and $Y$ are octahedral, then $X\oplus_\infty Y$ is octahedral.
\item[{\rm(d)}] If $X\oplus_\infty Y$ is octahedral, then $X$ is octahedral.
\end{itemize}
\end{proposition}
\begin{proof}
(a).
Assume that $X$ is octahedral. Let $(x_1,y_1),\dotsc,(x_n,y_n)\in S_{X\oplus_1 Y}$ and let $\eps>0$. By our assumption, there exists a $u\in S_X$ such that
\[
\|x_i+u\|\geq(1-\eps)\bigl(\|x_i\|+1\bigr)\quad\text{for all $i\in\{1,\dotsc,n\}$.}
\]
Hence, for all $i\in\{1,\dotsc,n\}$,
\[
\|(x_i,y_i)+(u,0)\|_1\geq(1-\eps)\bigl(\|x_i\|+1\bigr)+\|y_i\|\geq 2-2\eps.
\]

(b).
Let $x\in S_X$, $y\in S_Y$, and let $1<p<\infty$. We shall show that, for sufficiently small $\eps>0$, there is no $(u,v)\in S_{X\oplus_p Y}$ such that
\[
\bigl\|(x,0)+(u,v)\bigr\|_p\geq 2-\eps
\qquad\text{and}\qquad
\bigl\|(0,y)+(u,v)\bigr\|_p\geq 2-\eps.
\]
If such an element $(u,v)$ existed, then
\[
\bigl\|(x+u,y+v)\bigl\|_p^p=\|x+u\|^p+\|y+v\|^p\geq 2(2-\eps)^p-1.
\]
On the other hand,
\[\bigl\|(x+u,y+v)\bigl\|_p^p\le\bigl(\|(x,y)\|_p+\|(u,v)\|_p\bigr)^p=(2^{1/p}+1)^p.\]
For small $\eps$, we would have a contradiction  because
\[
2^{p+1}-1>(2^{1/p}+1)^p.
\]
The last inequality is easily obtained from the Minkowski's inequality by considering $(\sqrt[p]{2},0), (1,1)\in\mathbb R^2$.

(c).
Assume that $X$ and $Y$ are octahedral. Let $(x_1,y_1),\dotsc,(x_n,y_n)\in S_{X\oplus_\infty Y}$ and let $\eps>0$.
By our assumption, there are $u\in S_X$ and $v\in S_X$ such that
\[
\|x_i+u\|\geq (1-\eps)\bigl(\|x_i\|+1\bigr)\quad \text{for all $i\in\{1,\dotsc,n\}$,}
\]
and
\[
\|y_i+v\|\geq (1-\eps)\bigl(\|y_i\|+1\bigr)\quad \text{for all $i\in\{1,\dotsc,n\}$.}
\]
Consequently, for all $i\in\{1,\dotsc,n\}$,
\begin{align*}
\bigl\|(x_i,y_i)+(u,v)\bigr\|_\infty
&=\max\bigl\{\|x_i+u\|,\|y_i+v\|\bigr\}\\
&\geq (1-\eps)\bigl(\max\{\|x_i\|,\|y_i\|\}+1\bigl)\\
&=(1-\eps)\bigl(\|(x_i,y_i)\|_\infty+1\bigr).
\end{align*}

(d).
Assume that $X\oplus_\infty Y$ is octahedral.
Let $x_1,\dotsc,x_n\in S_X$ and let $0<\eps<1$. By our assumption, there exists a $(u,v)\in S_{X\oplus_\infty Y}$ such that
\[
\max\bigl\{\|x_i+u\|,\|v\|\bigr\}\geq 2-\eps\quad\text{for all $i\in\{1,\dotsc,n\}$.}
\]
Since $\|v\|\leq 1$, we have
\[
\|x_i+u\|\geq 2-\eps\quad\text{for all $i\in\{1,\dotsc,n\}$.}
\]
It follows that $\|u\|\geq 1-\eps$. Therefore, for all $i\in\{1,\dotsc,n\}$,
\[
\biggl\|x_i+\frac{u}{\|u\|}\biggr\|\geq \|x_i+u\|-\bigl(1-\|u\|\bigr)\geq 2-2\eps.
\]
\end{proof}

Proposition~\ref{prop: p-sum WOH} combined, respectively, with Theorems~\ref{thm: omnibus thm for the sd2P} and \ref{thm: omnibus thm for the weak* Sd2P} immediately gives the corresponding stability results for the strong diameter $2$ property and for the weak$^\ast$ strong diameter $2$ property.
These results for the strong diameter $2$ property are known, but our method provides an alternative approach.

\begin{corollary}
Let $X$ and $Y$ be Banach spaces.
\begin{itemize}
\item[{\rm(a)}]
If $X$ has the strong diameter $2$ property, then $X\oplus_\infty Y$ has the strong diameter $2$ property \emph{(see \cite[Proposition~4.6]{ALN})}.
\item[{\rm(b)}]

If $1<p<\infty$, then $X\oplus_p Y$ does not have the strong diameter $2$ property \emph{(cf. \cite[Theorem 3.2]{ABL}, \cite[Theorem~1]{HL})}.
%

\item[{\rm(c)}]
If $X$ and $Y$ have the strong diameter $2$ property, then $X\oplus_1 Y$ has the strong diameter $2$ property \emph{(see \cite[Theorem~2.7~(iii)]{ALN}, see also \cite[Lemma~2.1]{GP} and \cite[Proposition~3.1]{ABL})}.
\item[{\rm(d)}]
If $X\oplus_1 Y$ has the strong diameter $2$ property, then $X$ has the strong diameter $2$ property \emph{(see \cite[Proposition~3.1]{ABL})}.
\end{itemize}
\end{corollary}

\begin{corollary}
Let $X$ and $Y$ be Banach spaces.
\begin{itemize}
\item[{\rm(a)}] If $X^\ast$ has the weak$^\ast$ strong diameter $2$ property, then $(X\oplus_1 Y)^\ast$ has the weak$^\ast$ strong diameter $2$ property.
\item[{\rm(b)}] If $1<p<\infty$, then $(X\oplus_p Y)^\ast$ does not have the weak$^\ast$ strong diameter $2$ property.
\item[{\rm(c)}] If $X^\ast$ and $Y^\ast$ have the weak$^\ast$ strong diameter $2$ property, then $(X\oplus_\infty Y)^\ast$ has the weak$^\ast$ strong diameter $2$ property.
\item[{\rm(d)}]
If $(X\oplus_\infty Y)^\ast$ has the weak$^\ast$ strong diameter $2$ property, then $X^\ast$ has the weak$^\ast$ strong diameter $2$ property.
\end{itemize}
\end{corollary}

%
%
We denote the \emph{annihilator} of a subspace $Y$ of a Banach space $X$ by
\[
Y^{\perp}=\{\xast\in \Xast \colon \xast(y)=0\quad \text{for all } y\in Y\}.
\]

A closed subspace $Y$ of a Banach space $X$ is called an \emph{$M$-ideal} in~$X$ (see, e.g. \cite{HWW}) if there exists a norm-$1$ projection $P$ on $\Xast$ with $\ker P=Y^{\perp}$ and
\[
\n{\xast}=\n{P\xast}+\n{\xast-P\xast}\qquad \text{for all $\xast\in \Xast$.}
\]
If, in addition, the range $\ran P$ of $P$ is \emph{$1$-norming,}
i.e.
\[
\|x\|=\sup\bigl\{|\xs(x)|\colon\,\xs\in\ran P,\,\|\xs\|\leq1\bigr\}\quad\text{for all $x\in X$,}
\]
then $Y$ is called a \emph{strict $M$-ideal.}

Relations between $M$-ideal structure and diameter $2$ properties were first considered in \cite{P} where it was proven that if a proper subspace $Y$ of a Banach space $X$ is a strict $M$-ideal in $X$, then both $Y$ and $X$ have the diameter $2$ property (see \cite[Theorem~2.4]{P}). In \cite[Theorem~4.10]{ALN}, it is shown that, under the same assumptions, one can conclude that both $Y$ and $X$ have even the strong diameter $2$ property. (An immediate corollary of this is that if a nonreflexive Banach space $X$ is an $M$-ideal in its bidual $\Xss$, then both $X$ and $X^{\ast\ast}$ have the strong diameter $2$ property.) In Theorem \ref{prop: strict M-ideal}, we shall present a simple proof of this result.

In \cite{HL}, it is shown that if an $M$-ideal $Y$ in $X$ has some diameter~$2$ property, then $X$ has the same diameter $2$ property
without the assumption that the range of the $M$-ideal projection is $1$-norming.
The duality between diameter $2$ properties and octahedrality implies a very quick proof of this result.

\begin{proposition}[see {\cite[Propositions~3, 4, and 5]{HL}}]\label{prop: stability thm for M-ideals}
Let $X$ be a Banach space and let $Y$ be an $M$-ideal in~$X$.
\begin{itemize}
\item[{\rm(a)}]
If $Y$ has the local diameter $2$ property, then also $X$ has the local diameter $2$ property.
\item[{\rm(b)}]
If $Y$ has the diameter $2$ property, then also $X$ has the diameter $2$ property.
\item[{\rm(c)}]
If $Y$ has the strong diameter $2$ property, then also $X$ has the strong diameter $2$ property.
\end{itemize}
\end{proposition}
%

\begin{proof}
Since $Y$ is an $M$-ideal in $X$,  one has $\Xs=\ran P\oplus_1 \ker P$, where $P\colon\,\Xs\to\Xs$ is the $M$-ideal projection.
Since $\ran P$ is isometrically isomorphic to $\Ys$, the assertions (a), (b), and (c) follow, respectively,
from Theorem \ref{thm: omnibus thm for the ld2P} combined with Proposition \ref{prop: p-sum LOH}, (a),
from Theorem \ref{thm: omnibus thm for the d2P} combined with Proposition \ref{prop: p-sum WOH}, (a), and
from Theorem \ref{thm: omnibus thm for the sd2P} combined with Proposition \ref{prop: p-sum OH}, (a).
\end{proof}

\begin{proposition}[cf. {\cite[Theorem~4.10]{ALN}}]\label{prop: strict M-ideal}
Let $X$ be a Banach space and let a proper subspace $Y$ be a strict $M$-ideal in $X$.
Then both $Y$ and $X$ have the strong diameter $2$ property.
\end{proposition}
\begin{proof}
Letting $P\colon\,\Xs\to\Xs$ be the $M$-ideal projection, throughout the proof, for convenience, we identify $\ran P$ and $\Ys$ ``in the usual way''.

By Proposition \ref{prop: stability thm for M-ideals}, it suffices to show that $Y$ has the strong dia\-meter $2$ property.
To this end, letting $\ys_1,\dotsc,\ys_n\in S_{\Ys}$ ($n\in\N$) and $\eps>0$ be arbitrary,
by Theorem~\ref{thm: omnibus thm for the sd2P} and Proposition~\ref{prop: reformulations of oct}, it suffices to find a $\ys\in S_{\Ys}$ such that
\[
(1+\eps)\|\ys_j+\ys\|\geq2-7\eps
\quad\text{for all $j\in\{1,\dotsc,n\}$.}
\]
Choose an $x\in S_X$ so that $d(x,Y)>1-\eps$, and $y_1,\dotsc,y_n\in S_{Y}$ so that $\ys_j(y_j)>1-\eps$.
By \cite[Proposition~2.3]{DW}, there is a $z\in B_{Y}$ such that
\[
|\ys_j(x-z)|<\eps
\quad\text{and}\quad
\|\pm y_j+x-z \|<1+\eps
\quad\text{for all $j\in\{1,\dotsc,n\}$.}
\]
Let $\ys\in S_{\Ys}$ be such that
\[
\ys(x-z)>\|x-z\|-\eps\geq d(x,Y)-\eps>1-2\eps.
\]
Whenever $j\in\{1,\dotsc,n\}$, one has $\ys(y_j)>-3\eps$ because
\begin{align*}
1+\eps
&>\ys(-y_j+x-z)>-\ys(y_j)+1-2\eps,
\end{align*}
thus
\begin{align*}
(1+\eps)\|\ys_j+\ys\|
&\geq(\ys_j+\ys)(y_j+x-z)\\
&=\ys_j(y_j)+\ys_j(x-z)+\ys(y_j)+\ys(x-z)\\
&>1-\eps-\eps-3\eps+1-2\eps\\
&=2-7\eps.
\end{align*}
\end{proof}
\bibliographystyle{amsplain}
\footnotesize
\bibliography{Bibliography}

{\providecommand{\noopsort}[1]{}}
\providecommand{\bysame}{\leavevmode\hbox to3em{\hrulefill}\thinspace}
\providecommand{\MR}{\relax\ifhmode\unskip\space\fi MR }
\providecommand{\MRhref}[2]{%
  \href{http://www.ams.org/mathscinet-getitem?mr=#1}{#2}
}
\providecommand{\href}[2]{#2}
\begin{thebibliography}{10}

\bibitem{ALN}
T.~{A}brahamsen, V.~{L}ima, and O.~{N}ygaard, \emph{{R}emarks on diameter 2
  properties}, J. Conv. Anal. \textbf{20} (2013), 439--452.

\bibitem{ABL}
M.~D. {A}costa, J.~{B}ecerra {G}uerrero, and G.~{L}\'{o}pez {P}\'{e}rez,
  \emph{{S}tability results of diameter two properties}, J. Conv. Anal. (to
  appear).

\bibitem{GP}
J.~{B}ecerra {G}uerrero and G.~{L}\'{o}pez {P}\'{e}rez, \emph{{R}elatively
  weakly open subsets of the unit ball in functions spaces}, J. Math. Anal.
  Appl. \textbf{315} (2006), 544--554.

\bibitem{BGLPRZoca}
J.~{B}ecerra {G}uerrero, G.~{L}\'{o}pez {P}\'{e}rez, and A.~{R}ueda {Z}oca,
  \emph{Octahedral norms and convex combination of slices in {B}anach spaces},
  preprint.

\bibitem{GPZ}
J.~{B}ecerra {G}uerrero, G.~{L}\'{o}pez {P}\'{e}rez, and A.~{R}ueda {Z}oido,
  \emph{{B}ig slices versus big relatively weakly open subsets in {B}anach
  spaces}, preprint.

\bibitem{D}
R.~Deville, \emph{{A} dual characterisation of the existence of small
  combinations of slices}, Bull. Austral. Math. Soc. \textbf{37} (1988),
  113--120.

\bibitem{DGZ}
R.~Deville, G.~Godefroy, and V.~Zizler, \emph{{S}moothness and renormings in
  {B}anach spaces}, Pitman Monographs and Surveys in Pure and Applied
  Mathematics, vol.~64, Longman Scientific \& Technical, Harlow, 1993.

\bibitem{GGMS}
N.~{G}houssoub, G.~{G}odefroy, B.~{M}aurey, and W.~{S}chachermayer,
  \emph{{S}ome topological and geometrical structures in {B}anach spaces}, Mem.
  Amer. Math. Soc. \textbf{378} (1987).

\bibitem{G}
G.~Godefroy, \emph{{M}etric characterization of first {B}aire class linear
  forms and octahedral norms}, Studia Math. \textbf{95} (1989), 1--15.

\bibitem{GM}
G.~Godefroy and B.~Maurey, \emph{{N}ormes lisses et normes anguleuses sur les
  espaces de {B}anach s\'eparables}, preprint.

\bibitem{HL}
R.~{H}aller and J.~{L}angemets, \emph{{T}wo remarks on diameter $2$
  properties}, Proc. Estonian Acad. Sci. (to appear).

\bibitem{HWW}
P.~{H}armand, D.~{W}erner, and W.~{W}erner, \emph{{M}-ideals in {B}anach spaces
  and {B}anach algebras}, Lecture Notes in Mathematics, vol. 1547, Springer,
  Berlin, 1993.

\bibitem{P}
G.~{L}\'{o}pez {P}\'{e}rez, \emph{The big slice phenomena in {M}-embedded and
  {L}-embedded spaces}, Proc. Amer. Math. Soc. \textbf{134} (2005), 273--282.

\bibitem{DW}
D.~{W}erner, \emph{{M}-ideals and the ``basic inequality''}, J. Approx. Theory
  \textbf{76} (1994), 21--30.

\end{thebibliography}
\end{document}